\documentclass[12pt]{amsart}
\usepackage[all, cmtip]{xy}

\usepackage{times,amsfonts,amsmath,amstext,amsbsy,amssymb, amsopn,amsthm,upref,eucal, mathrsfs}

\newcommand \Conf {{\mathrm {Conf}}}

\newcommand \tr {{\mathrm {tr}}}

\newcommand \Prob {{\mathbb P}}

\newcommand{\ee}{\mathbb E}

\newcommand \DD {\mathbb{D}}

\newcommand{\E}{\mathbb E}

\newtheorem{theorem}{Theorem}[section]

\newtheorem{corollary}[theorem]{Corollary}
\newtheorem{proposition}[theorem]{Proposition}

\def\det{\qopname\relax o{det}}
\title[Conditional measures for the process with the Bergman kernel]
{The conditional measures for the determinantal point process with the Bergman kernel} 
\author{Alexander I. Bufetov}
\address{
CNRS, Aix-Marseille Universit{\'e}, Centrale Marseille,  Institut de Math{\'e}matiques de Marseille, UMR7373, 
  39 Rue F. Joliot Curie 13453, Marseille, France;\newline
  Steklov  Mathematical Institute of RAS, Moscow, Russia;\newline
    Institute for Information Transmission Problems, Moscow, Russia.  }
  \email{alexander.bufetov@univ-amu.fr, bufetov@mi-ras.ru}
  
\begin{document}
\maketitle
\begin{abstract} 
This note gives an explicit description of conditional measures for the determinantal point process with the Bergman kernel.
\end{abstract} 

\section{Introduction}

\subsection{Formulation of the main result}

The aim of this note is to give an explicit formula for the conditional measures of the zero set of the Gaussian Analytic Function under the condition that the configuration be fixed in the complement of a compact set. Recall that, by the Peres-Vir\`ag theorem \cite{PV-acta}, the zero set of the Gaussian Analytic Function is a determinantal point process with the Bergman kernel, cf.  \eqref{def-k} below. The main tool is the explicit representation obtained in \cite{BQ-holo} of the Radon---Nikodym derivative of the reduced Palm measure of our determinantal point process with respect to the process itself; the Radon---Nikodym derivative is found as a generalized multiplicative functional corresponding to the divergent Blaschke product over the particles of our configuration.

We proceed to the precise formulations. Let $a_n(\omega)$, $n\ge 0$, be independent standard complex Gaussian random variables, with expectation~$0$ and variance~$1$. The power series
\begin{equation}\label{GAF}
	\sum_{n=0}^\infty a_n(\omega)z^n
\end{equation}
almost surely has radius of convergence~$1$; the zero set of the series \eqref{GAF} is a configuration on the unit disc whose law $\mathbb{P}_K$, by the Peres---Vir\`ag theorem \cite{PV-acta},  is the determinantal measure on the space $\mathrm{Conf}(\mathbb{D})$ of configurations on $\mathbb{D}$ corresponding to the Bergman kernel
\begin{equation}\label{def-k}
	K(z,w)=\frac{1}{\pi(1-z\bar w)^2},\quad z,w\in\mathbb{D}
\end{equation}
of orthogonal projection in the space $L_2(\DD)$ of square-integrable functions with respect to the usual Lebesgue measure 
onto the closed subspace of square-integrable holomorphic functions.
 
Consider a decomposition $\mathbb{D}=B\sqcup C$ of the unit disc into two disjoint Borel sets with $B$ open and having compact closure in $\mathbb{D}$. The natural restriction map $\pi_C\colon X\mapsto X\cap C$ sends the measure $\mathbb{P}_K$ forward to its projection $\overline{\mathbb{P}_K^C}$; the $\overline{\mathbb{P}_K^C}$-almost surely defined conditional measures of $\mathbb{P}_K$ for a configuration $Y$ on $\mathrm{Conf}(\mathbb{D})$ satisfying $Y=Y\cap C$ on the preimage $\pi_C^{-1}(Y)$ is denoted $\mathbb{P}(\,\cdot\,|C;Y)$.  Lemma 1.11 in \cite{BQS} states that if $\Prob$ is a determinantal point process induced by a positive Hermitian contraction, then so is its conditional measure;  Lemma 1.11 also gives a limit procedure for finding the kernel governing the conditional measure. Our aim in this note is to give an explicit 
formula, see \eqref{main-formula} below, for the conditional measure $\mathbb{P}(\,\cdot\,|C;Y)$.
The starting point for the argument is Theorem 1.4 in \cite{BQ-holo} that gives an explicit expression for the Radon-Nikodym derivative of the Palm measure $\Prob_K^q$ of $\Prob_K$ with respect to $\Prob_K$; the Radon-Nikodym derivative, cf. \eqref{berg-palm} below, is expressed in terms of a regularized multiplicative functional $\overline \Psi_q$,  cf. \eqref{psi-bar} below, that we now write in a slightly different way.
The next step is an expression of the conditional measure $\mathbb{P}(\,\cdot\,|C;Y)$  in terms of the multiplicative functional $\overline \Psi_q$.

We consider the unit disc as the Poincar\'e model for the Lobachevsky plane, and for $p\in\mathbb{D}$, $R>0$ we let $D(p,R)$ stand for the Lobachevskian ball of radius~$R$ centred at $p$.

\begin{proposition}\label{reg-mult-gamma}
	For $\mathbb{P}_K$-almost every $X\in\mathrm{Conf}(\mathbb{D})$ and any $q\in\mathbb{D}$, the limit 
	\begin{multline}\label{def-psi-q}
		\tilde{\Psi}_q(X)=\\=
		\lim_{R\to\infty}\prod_{x\in X\cap D(q,R)}\biggl|\frac{x-q}{1-\bar{x}q}\biggr|^2\cdot
		\exp\biggl(\frac{\sqrt{-1}}{2\pi}\int\limits_{D(q,R)} \biggl(1-\biggl|\frac{z-q}{1-\bar{z}q}\biggr|^2\biggr)\frac{dz\wedge d\bar{z}}{(1-|z|^2)}\biggr) 
	\end{multline}
	exists in $L_1(\mathrm{Conf}(\mathbb{D}),\mathbb{P}_K)$ as well as $\mathbb{P}_K$-almost surely along a subsequence.
	
	We have
	\begin{equation*}
		\int \tilde{\Psi}_q(X)\,d\mathbb{P}_K(X)=\frac{e^{\gamma-1}}2,
	\end{equation*}
	where
	\begin{equation} \label{euma}
		\gamma=\lim_{n\to\infty}\Bigl(1+\frac{1}{2}+\dots+\frac{1}{n}-\log n\Bigr)
	\end{equation}
	is the Euler---Mascheroni constant.
\end{proposition}

Denote \begin{equation}\label{psi-bar}
\overline{\Psi}_q(X)=2e^{1-\gamma}\tilde{\Psi}_q(X).
\end{equation}
 We are now ready to proceed to the formulation of the main result of this note, an explicit description of the conditional measures of the determinantal point process with the Bergman kernel.

\begin{theorem}\label{main-thm}
	For $\overline{\mathbb{P}_K^C}$-almost every configuration $Y$, the conditional measure $\mathbb{P}_K(\,\cdot\,|C;Y)$ has the form
	\begin{equation}\label{main-formula}
		\eta_{Y,0}\biggl(
		1+\sum_{m=1}^\infty \det L_Y(q_j,q_k)_{j,k=1,\dots, m}\cdot
		\prod_{j=1}^m\biggl(\frac{\sqrt{-1}}{2\pi}\frac{dq_j\wedge d\bar{q}_j}{(1-|q_j|^2)^2}\biggr)
		\biggr),
	\end{equation}
	where
	\begin{equation}\label{l-def}
		L_Y(q_1,q_2)=\frac{\overline{\Psi}_{q_1}(Y)\overline{\Psi}_{q_2}(Y)}{1-q_1\bar{q}_2}
	\end{equation}
	and
	\begin{equation*}
		\eta_{Y,0}=\frac{1}{\det(1+L_Y)}=\mathbb{P}(\#_B=0\mid C;Y)
	\end{equation*}
	is the conditional probability that there are no particles in $B$.
\end{theorem}

Equivalently, the conditional measure is an L-ensemble in the sense of Borodin \cite{borodin}, and the L-kernel is given by the formula \eqref{l-def}.

\subsection{Palm Measures of the Point Process $\Prob_K$.}
\subsubsection{Correlation functions of point processes}
Let $E$ be a Polish space.
A {\it configuration} on $E$ is a countable or finite collection  of points in $E$, called {\it particles},  considered without regard to order and subject to  the additional requirement that every compact set contain only finitely many particles of a configuration.
Let $\Conf(E)$ be the space of configurations on $E$.
For a bounded Borel set $B\subset E$,   let
$$\#_B\colon\Conf(E)\to\mathbb{N}\cup\{0\}$$ be
the function that to a configuration
assigns the number of its particles
belonging to~$B$. The random variables $\#_B$ over all  bounded Borel sets $B\subset E$
determine the  Borel sigma-algebra on $\Conf(E)$.
A Borel probability measure $\Prob$ on $\Conf(E)$ is called {\it a point process} with phase space $E$.
Recall that the point process $\Prob$ is said to admit correlation measures
of order $l$ if for any continuous compactly supported function $\varphi$ on $E^l$
the functional
$$
\sum\limits_{x_1, \dots, x_l\in X} \varphi(x_1, \dots, x_l)
$$
is $\Prob$-integrable; the sum is taken over all ordered $l$-tuples of distinct particles in $X$. The $l$-th correlation measure $\rho_l$ of the point process $\Prob$ is then
defined by the formula
$$
\ee_{\Prob} \left(\sum\limits_{x_1, \dots, x_l\in X} \varphi(x_1, \dots, x_l)\right)=
\displaystyle \int\limits_{E^l} \varphi(q_1, \dots, q_l)d\rho_l(q_1, \dots, q_l).
$$
If all correlation measures of a point process are well-defined and for any $m\in\mathbb{N}$ the $m$-th correlation measure is absolutely continuous with respect to $m$-th tensor power of the first one, then say that our point process admits correlation functions of all orders.

\subsubsection{Campbell  and Palm Measures.}
Following Kallenberg \cite{kallenberg}, Daley--Vere-Jones \cite{DVJ}, we recall the definition of Campbell 
measures of point processes; the notation follows  \cite{buf-aop}.
Let $\Prob$ be a point process on $E$  admitting the first correlation measure 
$\rho_1^{\Prob}$.
The {\it Campbell measure}  ${\EuScript C}_{\Prob}$ of  $\Prob$ 
is a sigma-finite measure on $E\times \Conf(E)$ such that for any Borel subsets
$B\subset E$, ${\mathscr Z}\subset \Conf(E)$ we have
$$
{\EuScript C}_{\Prob}(B\times {\mathscr Z})=\displaystyle \int\limits_{{\mathscr Z}} \#_B(X)d\Prob(X).
$$

The Palm measure ${\hat \Prob}^q$ is the canonical conditional measure, in the sense of Rohlin \cite{Rohmeas}, 
of the Campbell measure ${\mathcal C}_{\Prob}$
with respect to the measurable partition of the space $E\times \Conf(E)$
into subsets  $\{q\}\times \Conf(E)$, $q\in E$, cf. \cite{buf-aop}.
By definition, the Palm measure ${\hat \Prob}^{q}$ is supported on the subset of configurations containing a particle at  position $q$.
Removing these particles, one defines the {\it reduced}
Palm measure $\Prob^{q}$ as the push-forward of the Palm measure
${\hat \Prob}^{q}$ under the erasing map
$X\to X\setminus \{q\}$. 

Iterating the definition, one arrives at iterated Campbell, Palm and reduced Palm measures:
the  $r$-th {\it Campbell measure}  ${\EuScript C}^r_{\Prob}$ of  $\Prob$ 
is a sigma-finite measure on the product  $E\times\dots\times \E\times \Conf(E)$ of $r$ copies of $E$ and $\Conf(E)$ such that for any disjoint Borel subsets
$B_1, \dots, B_r\subset E$, ${\mathscr Z}\subset \Conf(E)$ we have
$$
{\EuScript C}_{\Prob}(B_1\times\dots\times B_r\times {\mathscr Z})=\displaystyle \int\limits_{{\mathscr Z}} 
\#_{B_1}(X)\dots \#_{B_r}(X)d\Prob(X).
$$ 
Given distinct $q_1, \dots, q_r\in E$, the Palm measure ${\hat \Prob}^{q_1, \dots, q_r}$ is the canonical conditional measure, in the sense of Rohlin \cite{Rohmeas}, 
of the Campbell measure ${\mathcal C}^r_{\Prob}$
with respect to the measurable partition of the space $E\times\dots\times E \times  \Conf(E)$
into subsets  $\{q_1, \dots, q_r\}\times \Conf(E)$, $q\in E$, cf. \cite{buf-aop}.
By definition, the Palm measure ${\hat \Prob}^{q_1, \dots, q_r}$ is supported on the subset of 
configurations containing a particle at  each position $q_1, \dots, q_r$.
Removing these particles, one defines the {\it reduced}
Palm measure $\Prob^{q_1, \dots, q_r}$ as the push-forward of the Palm measure
${\hat \Prob}^{q}$ under the erasing map
$X\to X\setminus \{q_1, \dots, q_r\}$;
see Kallenberg \cite{kallenberg}, whose formalism is also adopted in  \cite{buf-aop}, for a more detailed exposition. 
As all conditional measures, reduced Palm measures $\Prob^q$ are {\it a priori} only defined  for $\rho_1$-almost every 
$q$. In our context of determinantal point processes, for any distinct $q_1, \dots, q_m\in E$, 
the Shirai-Takahashi Theorem  allows us to fix a convenient explicit Borel realization $\Prob^{q_1, \dots, q_m}$ of reduced Palm measures. 

\subsubsection{Determinantal Point Processes}
As before, let $E$ be a Polish space, and let $\mu$ be a sigma-finite Borel  measure on $E$. Recall that a Borel probability measure $\mathbb{P}$ on
$\Conf(E)$ is called
\textit{determinantal} if there exists a locally trace class operator  $K$ acting in $L_2(E, \mu)$   such that for any bounded measurable
function $g$, for which $g-1$ is supported in a bounded set $B$,
we have
\begin{equation}
\label{eq1}
\mathbb{E}_{\mathbb{P}}\prod\limits_{x\in X}g(x)
=\det\biggl(1+(g-1)K\chi_{B}\biggr).
\end{equation}
Here and elsewhere in similar formulas, $1$ stands for the identity operator.
The Fredholm determinant in~\eqref{eq1} is well-defined since
$K$ is locally of trace class.
The equation (\ref{eq1}) determines the measure $\Prob$ uniquely. We use the notation $\Prob_K$ for the determinantal measure 
induced by the operator $K$.
By a theorem due to Macch{\` \i} and Soshnikov ~\cite{Macchi}, \cite{soshnikov} and Shirai-Takahashi \cite{ShirTaka0}, any
Hermitian positive contraction that belongs
to the local trace class defines a determinantal point process.

\subsubsection{Generalized multiplicative functionals}	

Let $g$ be a Borel function on $E$ and let  $\Psi$ be a Borel function defined on a Borel subset $\mathscr{Z}\subset \mathrm{Conf}(E)$ and satisfying the following: if $X,Y\in \mathscr{Z}$ and there exist distinct particles
$p_1,\dots, p_r, q_1,\dots, q_s\in E$ such that 
$X\setminus\{p_1,\dots,p_r\}=Y\setminus\{q_1,\dots,q_s\}$, then
\begin{equation}\label{reg-mult}
	\Psi(X)=\frac{g(p_1)\cdots g(p_r)}{g(q_1)\cdots g(q_s)}\Psi(Y).
\end{equation}
In this case we say that $\Psi$ is a generalized multiplicative functional corresponding to the function $g$. 
The regularized multiplicative functional \eqref{def-psi-q} is a particular case of a generalized multiplicative functional.
If the point process $\mathbb{P}$ has trivial tail $\sigma$-algebra, then a generalized multiplicative functional corresponding to a function $g$, provided  it exists,  is $\mathbb{P}$-almost surely unique up to multiplication by a constant.
\subsubsection{The Characterization of Palm Measures for $\mathbb{P}_K$}
The starting point for the argument is Theorem 1.4 in \cite{BQ-holo} that, in view of Proposition \ref{reg-mult-gamma}, can be formulated as follows:  for any $q\in\DD$,  
 the reduced Palm measure  $\mathbb{P}_K^q$ of our determinantal point process $\Prob_K$ with the Bergman kernel is given by the formula
\begin{equation}\label{berg-palm}
\mathbb{P}_K^q=\overline \Psi_q \mathbb{P}_K.
\end{equation} 
 In other words, the Radon-Nikodym derivative is given by a regularized  multiplicative functional. The argument is completed by a general proposition describing the conditional measures for a point process whose Palm measures are expressed as a product of the original measure and a multiplicative functional.
 
 \subsection{Palm Measures, Multiplicative Functionals and Conditional Measures.}
 Let $\Prob$ be a point process with phase space $E$ admitting correlation functions of all orders. We fix a  Borel realization $\Prob^q$ of its reduced Palm measures and assume  that there exists a symmetric positive Borel function $B(q_1, q_2)$, $q_1, q_2\in E$, defined on $E\times E$ and such that for any $q\in E$ the Radon---Nikodym derivative $d\mathbb{P}^q/d\mathbb{P}$ is a generalized multiplicative functional corresponding to the function $B(q,\,\cdot\,)$.
	
\begin{proposition} \label{gen-mult-palm}
For any decomposition $E=B\sqcup C$ into two Borel sets with $\rho_1^{\Prob}(B)<+\infty$  and for $\overline{\mathbb{P}_C}$-almost every $Y\in\mathrm{Conf}(E;C)$ the conditional measure $\mathbb{P}_Y=\mathbb{P}(\,\cdot\,|Y,C)$ has the form
	\begin{equation*}
		\eta_{Y,0}\biggl(1+\sum_{m=1}^\infty\prod_{1\le i<j\le m} B(q_i,q_j)\cdot
		\prod\limits_{i=1}^{m}\frac{d\mathbb{P}^{q_i}}{d\mathbb{P}}(Y)\,d\rho_1(q_i)\biggr),
	\end{equation*}
	where $\eta_{Y,0}=\mathbb{P}(\#_B(X)=0|Y;C)$.
\end{proposition}

Proposition \ref{gen-mult-palm}, together with Proposition  \ref{reg-mult-gamma}, directly implies Theorem \ref{main-thm}
in view of the Cauchy identity
$$
\displaystyle \det \left(\displaystyle \frac 1{1-q_j \overline q_k}\right)_{j,k=1, \dots, n} \ = \  
\displaystyle\frac{\displaystyle\prod\limits_{1\leq j<k\leq n}\big|q_j-q_k\big|^2}{\displaystyle\prod\limits_{1\leq j<k\leq n}\big|1-q_j \overline q_k\big|^2}.
$$

It remains to prove Proposition  \ref{reg-mult-gamma} and Proposition \ref{gen-mult-palm}.

\section{Proof of Proposition \ref{gen-mult-palm}}
\subsection{Conditional measures of point processes.}
Let $E$ be a locally compact complete metric space, let $\Conf(E)$ be the space of configurations on $E$. Given a configuration $X \in \textrm{Conf}(E)$ and a subset $C \subset E$, we let $X|_C$ stand for the restriction of $X$ onto the subset $C$.  We let $\overline \Prob^C$ be the push-forward measure under the natural projection $X \to X|_C$,
Given a point process on $E$, that is,  a Borel probability measure $\Prob$ on $\Conf(E)$, the measure $\mathbb{P}(\cdot | X; C)$ on $\Conf(E\setminus C)$ is defined as the conditional measure of $\mathbb{P}$ with respect to the condition that the restriction of our random configuration onto $C$ coincides with $X|_C$. More formally, we consider the surjective restriction mapping $X \to X|_C$ from
$\textrm{Conf}(E)$ to $\textrm{Conf}(C)$. Fibres of this mapping can be identified with $\mathrm{Conf}(E\backslash C)$ and conditional measures, in the sense of Rohlin \cite{Rohmeas},  are precisely the measures $\mathbb{P}(\cdot | X; C)$. 
Let $\Conf(E; C)$ be the subset of those configurations on $E$ all whose particles lie in $C$; in other words, the image of the  
natural projection $X \to X|_C$.
By definition, we have 
$$
\Prob=\int\limits_{\Conf(E; C)} \mathbb{P}(\cdot | Y; C) d\overline \Prob^C(Y).
$$

 The decomposition into conditional measures is by definition lifted onto the level of Campbell measures:
\begin{equation}\label{campbell-lift}
{\EuScript C}_{\Prob}=\int\limits_{\Conf(E; C)} {\EuScript C}_{\mathbb{P}(\cdot | Y; C)} d\overline \Prob^C(Y)
\end{equation}

\subsection{Palm measures of different orders}
As before, let $\mathbb{P}$ be a point process on a Polish space $E$.  We assume that the point process $\Prob$ admits correlation functions of all orders and that the reduced Palm measures of $\Prob$ are almost surely absolutely continuous with respect to $\Prob$; it follows that reduced Palm measures of all orders are also almost surely absolutely continuous with respect to $\Prob$; almost surely is here understood with respect to the first correlation measure. 
It is convenient to think that the space $E$ is endowed with a sigma-finite Borel measure $\mu$ such that the first correlation measure
of $\Prob$ is absolutely continuous with respect to $\mu$; the $m$-th correlation measure of $\Prob$ then has the form 
$\rho_m(q_1, \dots, q_m) d\mu(q_1)\dots d\mu(q_m)$, where $\rho_m$ is the $m$-th correlation function.

For $\mu$-almost any distinct points $p_1,\dots,p_m,q_1,\dots,q_r$ and $\mathbb{P}$-almost any configuration $X\in\mathrm{Conf}(E)$ not containing any of the points $p_1,\dots,p_m$, $q_1,\dots,q_r$ the following identity directly follows from the definition of the Palm measures:
\begin{multline}\label{add-part}
	\frac{\rho_{m+r}(p_1,\dots,p_m,q_1,\dots,q_r)}{\rho_{r}(q_1,\dots,q_r)}\cdot
	\frac{d\mathbb{P}^{p_1,\dots,p_m,q_1,\dots,q_r}}{d\mathbb{P}^{q_1,\dots,q_r}}(X)={}\\
	{}=	\rho_{m}(p_1,\dots,p_m)
	\frac{d\mathbb{P}^{p_1,\dots,p_m}}{d\mathbb{P}}(X,q_1,\dots,q_r),
\end{multline}

\subsection{Palm measures and conditional measures}
As before, we consider a point process $\Prob$ on the phase space $E$; the point process $\Prob$ is assumed to admit correlation functions of all orders.
Consider a decomposition  $E=B\sqcup C$ of our phase space $E$ as a disjoint union of two Borel sets.
As before, we let $\overline \Prob^C$ be the push-forward measure under the natural projection $X \to X|_C$, and, for a configuration $Y$ all whose particles lie in $C$, in this subsection, we write $\mathbb{P}_{[Y, C]}=\mathbb{P}(\,\cdot\,|Y;C)$. We take a natural $m$ and let $\Prob$ be a point process whose reduced Palm measures of order $m$ are $\rho_m^{\Prob}$-almost surely absolutely continuous with respect to $\Prob$.
 From \eqref{campbell-lift} it follows that,   $\overline \Prob^C$-almost surely, the $m$-th correlation measure of the measure $\mathbb{P}_{[Y, C]}$ is absolutely continuous 
with respect to the $m$-th tensor power of the first correlation measure of $\Prob$ (note that for determinantal point processes governed by Hermitian contractions
this requirement is automatically verified by Lemma 1.11 of \cite{BQS} on the preservation of the determinantal property under taking conditional measures; observe, however, that  Lemma 1.11 is not used in this derivation of the explicit form of the conditional measures).
Let 
$\rho_{{[Y, C]}, m}$ be the $m$-th correlation function of $\mathbb{P}_{[Y, C]}$.
By definition, we  have
\begin{equation}\label{palm-indentity-cond}
	\frac{\rho_{{[Y, C]},m}(q_1,\dots,q_m)\,d\mathbb{P}_{[Y, C]}^{q_1,\dots,q_m}}{d\mathbb{P}_{[Y, C]}}(Z)=
	\frac{\rho_m(q_1,\dots,q_m)d\mathbb P^{q_1,\dots,q_m}}{d\mathbb{P}}(Y\cup Z).
\end{equation}
 
We now directly obtain the following

\begin{corollary}Assume that the  first correlation measure of the set $B$ is finite.
Then for $\overline \Prob ^ C$-almost  any configuration $Y$ the conditional measure $\mathbb{P}_{[Y, C]}$ has the form
\begin{equation}\label{cond-bdd}
	\sum_{m=0}^\infty \eta_{{[Y, C]},m}(q_1,\dots,q_m)\,d\mu(q_1)\dots d\mu(q_m),
\end{equation}
where $\eta_{{[Y, C]},0}=\mathbb{P}_{[Y, C]}(\varnothing)$ is the conditional probability of the absence of particles in $B$ and
\begin{equation}\label{eta-char}
	\eta_{{[Y, C]},m}(q_1,\dots,q_m)=\eta_{{[Y, C]},0}\cdot 
	\frac{\rho_m(q_1,\dots,q_m)d\mathbb P^{q_1,\dots,q_m}}{d\mathbb{P}}(Y).
\end{equation}
\end{corollary}

     \begin{proof} That the conditional measure $\mathbb{P}_{[Y, C]}$ is absolutely continuous with respect to the Poisson process of intensity $\mu$ follows from the fact that the measure $\mathbb{P}_{[Y, C]}$ is, $\overline \Prob_C$-almost surely, supported on the set  of configurations with finitely many particles  and that the $m$-th correlation measure of the measure $\mathbb{P}_{[Y, C]}$ is absolutely continuous 
with respect to the $m$-th tensor power of the first correlation measure of $\Prob$.
 Now take a general measure of the form \eqref{cond-bdd} for some Borel functions $\eta_{{[Y, C]},m}$, $m=1, \dots$; 
 the $r$-th Palm measure of our measure at the particles $p_1, \dots, p_r$  takes the form
 $$
 M^{-1}(p_1, \dots, p_r)\sum_{m=0}^\infty \eta_{{[Y, C]},m+r}(p_1, \dots, p_r, q_1,\dots,q_m)\,d\mu(q_1)\dots d\mu(q_m),
 $$
  where    $M(p_1, \dots, p_r)$ is a normalization constant. We now write \eqref{palm-indentity-cond} with $Z=\varnothing$. By definition, $$\rho_{{[Y, C]},m}(q_1,\dots,q_m)\cdot \mathbb{P}_{[Y, C]}^{q_1,\dots,q_m}(\varnothing)=\eta_{{[Y, C]},m}(q_1,\dots,q_m),$$ and the desired equality \eqref{eta-char} follows. 
 \end{proof}

\subsection{Conclusion of the proof of Proposition \ref{gen-mult-palm}.}
Iterating \eqref{add-part}, we arrive at the identity
\begin{multline}\label{telescope}
	\rho_{m}(p_1,\dots,p_m)
	\frac{d\mathbb{P}^{p_1,\dots,p_m}}{d\mathbb{P}}(X)={}\\
	{}=\rho_1(p_1)\frac{d\mathbb{P}^{p_1}}{d\mathbb{P}}(X,p_2,\dots,p_m)\cdots
	\rho_1(p_{m-1})\frac{d\mathbb{P}^{p_{m-1}}}{d\mathbb{P}}(X,p_m)\cdot
	\rho_1(p_m)\frac{d\mathbb{P}^{p_m}}{d\mathbb{P}}(X).
\end{multline}
Combining with \eqref{eta-char}, we obtain the expression
\begin{multline*}
	\eta_{Y,m}(q_1,\dots,q_m)={}\\
	{}=\eta_{Y,0}\cdot \rho_1(q_1)\cdots\rho_m(q_m)\cdot
	\frac{d\mathbb{P}^{q_1}}{d\mathbb{P}}(Y)\cdot
	\frac{d\mathbb{P}^{q_2}}{d\mathbb{P}}(q_1,Y)\cdots
	\frac{d\mathbb{P}^{q_m}}{d\mathbb{P}}(q_1,\dots,q_{m-1},Y).
\end{multline*}
Substituting the expression of the Radon-Nikodym derivative as the multiplicative functional completes
the proof. \qed

\section{Proof of Proposition \ref{reg-mult-gamma}}
Recall that the Hilbert-Carleman regularization $\det_2$ of the Fredholm determinant is introduced on finite rank operators  
by the formula
\begin{equation}
\det_2(1+A)=\exp(-\tr A) \det(1+A)
\end{equation}
and then extended by continuity onto the space of  Hilbert-Schmidt operators.  

We first recall a well-known observation (cf. e.g. \cite{OS}). Let $g\colon\mathbb{D}\to\mathbb{R}_+$ be a nonnegative bounded Borel radial function, in other words, a function depending only on the absolute value of its argument. The eigenfunctions of the operator $\sqrt{g}K\sqrt{g}$ are precisely the functions $\sqrt{g} z^k$, $k\ge 0$, and the corresponding eigenvalue is
\begin{equation*}
	\frac{k+1}{\pi}\int_{\mathbb{D}}g(z)|z|^{2k}\,dz=
	(k+1)\int_0^1 \tilde{g}(\rho)\rho^k\,d\rho,
\end{equation*}
where $\tilde g(\rho)=g(\sqrt{\rho}e^{i\theta})$ for any $\theta$.

It directly follows that the operator
\begin{equation}\label{mult-op}
	K_1(z,w)=\sqrt{1-|z|^2}K(z,w)\sqrt{1-|w|^2}
\end{equation}
is Hilbert---Schmidt, as its eigenvalues are
\begin{equation*}
	(k+1)\int_0^1 (1-\rho)\rho^k\,d\rho=\frac{1}{k+2},\quad k=0,1,\dots.
\end{equation*}
By definition \eqref{def-psi-q}, we have
\begin{equation}\label{one-det}
	\mathbb{E}_{\mathbb{P}_K}\tilde{\Psi}=\det\nolimits_2(1+K_1).
\end{equation}
Note here that
\begin{equation*}
	\frac{1}{\pi(1-|z|^2)}=(1-|z|^2)K(z,z).
\end{equation*}
For $r\in(0,1)$, set 
\begin{equation*}
	\tilde{\Psi}_r(X)=\prod_{x\in X, |x|<r}|x|^2\cdot
	\exp\biggl(\frac{\sqrt{-1}}{2\pi}\int_{\{z:|z|<r\}}\frac{dz\wedge d\bar{z}}{(1-|z|^2)}\biggr).
\end{equation*}	
By definition, we have
\begin{equation*}
	\mathbb{E}_{\mathbb{P}_K}\tilde{\Psi}_r=\det\nolimits_2(1+\chi_{\{z:|z|<r\}}K_1\chi_{\{z:|z|<r\}}).
\end{equation*}
Since $$\chi_{\{z:|z|<r\}}K_1\chi_{\{z:|z|<r\}}\to K_1$$
in the Hilbert---Schmidt norm as $r\to 1$, 
writing the Cauchy-Bunyakovsky-Schwartz inequality
$$
\mathbb{E}_{\mathbb{P}_K} \left|\tilde{\Psi}- \tilde{\Psi}_r\right|\leq \sqrt{\mathbb{E}_{\mathbb{P}_K} 
\left|\tilde{\Psi}_r\right|^2}\sqrt{ \mathbb{E}_{\mathbb{P}_K} \left|\tilde{\Psi}\tilde{\Psi}_r^{-1}-1\right|^2}
$$
one directly checks the relation
$$
 \lim\limits_{r\to 1}\mathbb{E}_{\mathbb{P}_K} \left|\tilde{\Psi}- \tilde{\Psi}_r\right|=0,
 $$
 which is to say that the limit
\begin{equation*}
	\tilde{\Psi}=\lim_{r\to 1}\tilde{\Psi}_r
\end{equation*}
exists in $L_1(\mathrm{Conf}(\mathbb{D}),\mathbb{P}_K)$ as well as almost surely along a subsequence.

We now compute the right-hand side of \eqref{one-det}. In order to do so, we let $K^{(n)}$ be the orthogonal projection onto $\{1,z,\dots,z^n\}$ in $L_2(\mathbb{D})$ and write
\begin{equation*}
	K_1^{(n)}=\sqrt{1-|z|^2}K^{(n)}\sqrt{1-|w|^2}.
\end{equation*}

We have $K_1^{(n)}\to K_1$ in the Hilbert---Schmidt norm as $n\to\infty$, whence
\begin{equation*}
	\det\nolimits_2(1+K_1)=\lim_{n\to\infty} \det(1+K_1^{(n)})\times
	\exp\biggl(-\tr K_1^{(n)}\biggr).
\end{equation*}
By definition, we have
\begin{equation*}
	\det(1+K_1^{(n)})=\prod_{k=0}^n\Bigl(1+\frac{1}{k+2}\Bigr)=\frac{n+3}2.
\end{equation*}
and
\begin{equation*}
	\tr (K_1^{(n)})
	=\frac{1}{2}+\dots+\frac{1}{n+1}.
\end{equation*}
Summing up, cf. \eqref{euma}, we obtain
\begin{equation*}
	\det\nolimits_2(1+K_1)=\frac{e^{\gamma-1}}2.
\end{equation*}
We therefore obtain an alternative representation of the Palm measure $\mathbb{P}_K^0$ with respect to the original measure:
\begin{equation*}
	\frac{d\mathbb{P}_K^0}{d\mathbb{P}}(X)=
	\frac{e^{\gamma-1}}2\lim\limits_{r\to 1}\prod\limits_{x\in X:|x|<r} |x|^2 
	\exp\biggl(\frac{\sqrt{-1}}{2\pi}\int\limits_{\{z:|z|<r\}}\frac{dz\wedge d\bar z}{1-|z|^2}\biggr).
\end{equation*}
Setting $D(z,R)$ to be the Lobachevskian ball centred at $z$ and of Lobachevskian radius $R$, for any $q\in\mathbb{D}$ we rewrite
\begin{multline*}
	\frac{d\mathbb{P}_K^q}{d\mathbb{P}}(X)=
	\frac{e^{\gamma-1}}2\lim\limits_{R\to\infty}\prod\limits_{x\in D(q,R)\cap X}\biggl|\frac{x-q}{1-\bar{q}x}\biggr|^2\times \\
	\times\exp\biggl(\frac{\sqrt{-1}}{2}\int\limits_{D(q,R)}
	\biggl(1-\biggl|\frac{z-q}{1-\bar{q}z}\biggr|^2\biggr)\biggr)K(z,z)\,dz\wedge d\bar{z}.	
\end{multline*}

Proposition \ref{reg-mult-gamma} is proved.  Theorem \ref{main-thm} is proved completely. \qed

\noindent {\bf{Acknowledgements.}} I am deeply grateful to Dmitrii Khliustov, Alexey Klimenko and Yanqi Qiu for useful discussions. Part of this work was done during a visit to the Scuola Internazionale degli Studi Superiori Avanzati in Trieste and to 
the Alma Mater Studiorum--University of Bologna. I am deeply grateful to both institutions for their warm hospitality.
This research has received funding from the European Research Council (ERC) under the European Union's Horizon 2020 research and innovation programme, grant agreement No 647133 (ICHAOS), as well as from the  ANR grant  ANR-18-CE40-0035 REPKA.


\begin{thebibliography}{99}

\bibitem{borodin}
Alexei Borodin,  Determinantal point processes, Oxford Handbook of Random Matrix Theory, Edited by Gernot Akemann, Jinho Baik, and Philippe Di Francesco,  Oxford University Press, 2015, pp.231-250.

\bibitem{borodin-rains} 
A.M. Borodin, E.M. Rains, Eynard-Mehta theorem, Schur process, and their pfaffian analogs.
J. Stat. Phys. 121 (2005), 291--317.

\bibitem{buf-cond}
 A. I. Bufetov,
Conditional measures of determinantal point processes, arXiv:1605.01400, Functional Analysis and its Applications,
 54:1 (2020), 7–20. 

\bibitem{buf-aop}
A.~I. Bufetov.
 {Q}uasi-{S}ymmetries of {D}eterminantal {P}oint {P}rocesses, arxiv:1409.2068, 
  Ann. Probab. {46}(2018),  956-1003 .
 
 
 
 
  
 \bibitem{BQ-holo} 

A.~I. Bufetov and Y. Qiu,
   {D}eterminantal point processes associated with {H}ilbert spaces of holomorphic functions.
  Commun. Math. Phys., 351(2017), no.1, 1-44.

\bibitem{BQ-JFA} A.~I. Bufetov and Y. Qiu,
Conditional measures of generalized Ginibre point processes, 
Journal of Functional Analysis
272 (2017),  11, pp. 4671 -- 4708.


\bibitem{BQS} A.Bufetov, Y.Qiu, A. Shamov, Kernels of conditional determinantal measures, arxiv:1612.06751, to appear in the Journal of the European Mathematical Society, 2021, Volume 23, Issue 5, pp. 1477–1519.


\bibitem{DVJ} D.J.Daley, D. Vere-Jones, An introduction to the theory of point processes, vol.I-II, Springer Verlag 2008.

\bibitem{ghosh-compl}
S. Ghosh,  Determinantal processes and completeness of random exponentials: the critical case, 
Probability Theory and Related Fields, December 2015, Volume 163, Issue 3 -- 4, pp. 643 -- 665.

\bibitem{GP}
S. Ghosh, Y. Peres. Rigidity and tolerance in point processes: Gaussian zeros and Ginibre eigenvalues, 
Duke Math. J.,  Volume 166, Number 10 (2017), 1789 -- 1858.

\bibitem{G} S. Ghosh,
Rigidity and Tolerance in Gaussian zeros and Ginibre eigenvalues: quantitative estimates, arXiv:1211.3506

\bibitem{hol-soo}A.E. Holroyd, T. Soo,  Insertion and deletion tolerance for point processes,  
Electron. J. Probab. 18,  1 -- 24, 2013. 

\bibitem{kallenberg} O. Kallenberg,  Random Measures. Akademie-Verlag, Berlin,  1986.

\bibitem{lyons} R. Lyons, Determinantal probability measures.
Pub. Mat. IH\'ES. 98 (2003), 167--212.


\bibitem{Macchi} O. Macchi, The coincidence approach to stochastic point processes.
Advances in Appl. Probability, 7 (1975), 83--122.

\bibitem{GO-Adv} G. Olshanski,
The quasi-invariance property for the Gamma kernel determinantal measure. Advances in Mathematics,
2011. Vol. 226. P. 2305 -- 2350.

\bibitem{OS}
H. Osada and T. Shirai, 
 Absolute continuity and singularity of {P}alm measures of the
  {G}inibre point process.
  Probab. Theory Related Fields,  165 (2016), no. 3-4, 725--770. 

\bibitem{PV-acta}
Y. Peres and B. Vir{\'a}g.
Zeros of the i.i.d.\ {G}aussian power series: a conformally invariant
  determinantal process.
 Acta Mathematica, 194(1):1--35, 2005.


\bibitem {Rohmeas}
Rohlin, V. A. On the fundamental ideas of measure theory. (Russian) Mat. Sbornik N.S. 25(67), 
(1949),  107 -- 150.

\bibitem{ShirTaka0} T.Shirai, Y. Takahashi, Fermion process and Fredholm
determinant, Proceedings of the Second ISAAC
Congress, vol. I, 15 -- 23, Kluwer 2000.


\bibitem{ShirTaka1}T. Shirai, Y. Takahashi, Random point fields associated with certain Fredholm determinants.
I. Fermion, Poisson and boson point processes. J. Funct. Anal. 205 (2003), no. 2, 414--463.

\bibitem{ShirTaka2}T. Shirai, Y. Takahashi, Random point fields associated with certain Fredholm determinants.
II. Fermion shifts and their ergodic and Gibbs properties. Ann. Probab. 31 (2003), no. 3, 1533--1564.

\bibitem{Simon} B. Simon, Trace class ideals, AMS, 2011.


\bibitem{soshnikov} A. Soshnikov, Determinantal random point fields.
(Russian) Uspekhi Mat. Nauk 55 (2000), no. 5(335), 107--160;
transl. Russian Math. Surv. 55 (2000), no. 5, 923--975.


\end{thebibliography}
\end{document}